\documentclass[12pt]{amsart}
\usepackage{color}
\usepackage{amsmath}
\usepackage[colorlinks=true]{hyperref}

\usepackage{amscd,amsthm,amsfonts,amssymb,esint}
\usepackage{fullpage}
\usepackage[all]{xy}
\usepackage{mathtools}
\usepackage{xcolor}
\usepackage{mathrsfs}
\usepackage{amsmath}
\usepackage[all]{xy}
\usepackage{geometry}
\usepackage{array}
\usepackage{enumerate}
\usepackage{tikz-cd}
\usepackage{comment}
\usepackage{setspace}


\begin{document}
\begin{spacing}{1.25}
	%----------------------------------------

\newtheorem{theorem}{Theorem}[section]
\newtheorem{lemma}[theorem]{Lemma}
\newtheorem{proposition}[theorem]{Proposition}
\newtheorem{corollary}[theorem]{Corollary}
\newtheorem{conjecture}[theorem]{Conjecture}
\newtheorem{conv}[theorem]{Convention}
\newtheorem{ques}[theorem]{Question} 
\newtheorem{prb}[theorem]{Problem}

\newtheorem*{DML}{Dynamical Mordell--Lang Conjecture (DML Conjecture)}

	%----------------------------------------
	\theoremstyle{definition}
	\newtheorem{definition}[theorem]{Definition}
	\newtheorem{example}[theorem]{Example}
	
	%----------------------------------------
	\theoremstyle{remark}
	\newtheorem{remark}[theorem]{Remark}
	
	%----------------------------------------
	\numberwithin{equation}{section}

	%----------------------------------------

	\title{Dynamical Mordell--Lang conjecture for split self-maps of affine curve times projective curve}
	\author{Junyi Xie}
	\address{Beijing International Center for Mathematical Research \\ Peking University \\ Beijing 100871 \\ China}
	\email{xiejunyi@bicmr.pku.edu.cn}
	\author{She Yang}
	\address{Beijing International Center for Mathematical Research \\ Peking University \\ Beijing 100871 \\ China}
	\email{ys-yx@pku.edu.cn}
	\author{Aoyang Zheng}
	\address{Beijing International Center for Mathematical Research \\ Peking University \\ Beijing 100871 \\ China}
	\email{zay@pku.edu.cn}

	%----------------------------------------
	\begin{abstract}
We prove the dynamical Mordell--Lang conjecture for product of endomorphisms of an affine curve and a projective curve over $\overline{\mathbb{Q}}$.
	\end{abstract}

	%----------------------------------------
	\maketitle
	\setcounter{tocdepth}{1}
%	\tableofcontents
	
\section{Introduction}\label{Sec_intro}

The dynamical Mordell--Lang conjecture is one of the core problems in the field of arithmetic dynamics. It was proposed by Ghioca and Tucker in \cite{GT09} and can be stated as follows:

\begin{DML}
    Let $f:X\to X$ be an endomorphism of a quasi-projective variety over a field $K$ of characteristic 0, and $V$ be a closed subvariety of $X$. Then for every $x\in X(K)$, the return set $\{ n\in \mathbb{N}|\  f^n(x) \in V(K)\}$ is a finite union of arithmetic progressions.
\end{DML}

There is an extensive literature on various cases of the DML conjecture. Two significant cases are as follows:

\begin{enumerate}
    \item If $X$ is a quasi-projective variety over $\mathbb{C}$, and $f$ is an \'etale endomorphism of $X$, then the DML conjecture holds for $(X,f)$. See \cite{Bel06} and \cite[Theorem 1.3]{BGT10}.
    \item If $X = \mathbb{A}^2_{\mathbb{C}}$, and $f$ is an endomorphism of $X$, then the DML conjecture holds for $(X,f)$. See \cite{Xie17} and \cite[Theorem 3.2]{Xie}.
\end{enumerate}

One can consult \cite{BGT16,Xie} and the references therein for further known results.

In this article, we investigate the DML conjecture for certain types of split endomorphisms.

\begin{theorem}\label{mainthm}
    Let $X$ be an affine curve and $Y$ be a projective curve over $\overline{\mathbb{Q}}$. Let $f:X\to X$ and $g:Y\to Y$ be endomorphisms. Then DML conjecture holds for $(X\times Y, f\times g)$.
\end{theorem}

It turns out that major case is about $\mathbb{A}^1\times\mathbb{P}^1$. Hence we state this as a proposition and will mainly deal with it later.

\begin{proposition}\label{mainprop}
    Let $f: \mathbb{A}^1 \to \mathbb{A}^1$ and $g: \mathbb{P}^1 \to \mathbb{P}^1$ be endomorphisms over $\overline{\mathbb{Q}}$. Then DML conjecture holds for $(\mathbb{A}^1 \times \mathbb{P}^1, f \times g)$.
\end{proposition}

In the proof of Proposition \ref{mainprop}, we first apply the results in \cite{BGT10,Xie17} and \cite{XY} to make some reductions and further assumptions about the endomorphisms. Namely, we may assume that $\mathrm{deg}(f)=\mathrm{deg}(g)>1$ and no iterate of $g$ can conjugate to a polynomial. Then the key observation is that for an appropriate place, the $\mathbb{A}^1$ coordinate of a non-preperiodic orbit tends to infinity with the maximal speed, while the $\mathbb{P}^1$ coordinate of any subsequence of that orbit cannot tend to any point with such a speed. This forces that orbit to have a finite intersection with every (non-horizontal and non-vertical) curve.

We can deduce from Theorem \ref{mainthm} that the DML conjecture holds in more general settings.

We recall a concept following \cite[Definition 1.3]{Xie}. For a quasi-projective variety $X$ over a field $K$ and an endomorphism $f$, we say $(X, f)$ satisfies the \emph{DML(1) property}, if for any curve $C\subseteq X$ and any point $x\in X(K)$, the return set $\{ n\in \mathbb{N}|\ f^n(x)\in C(K)\}$ is a finite union of arithmetic progressions. Here ``1" stands for the dimension of the closed subvariety.

\begin{corollary}\label{cor2}
    Let $X$ be an affine variety and $Y$ be a projective variety over $\overline{\mathbb{Q}}$. Let $f:X \to X$ and $g:Y\to Y$ be dominant endomorphisms. Assume that $(X, f)$ and $(Y, g)$ satisfy the DML(1) property. Then $(X\times Y, f\times g)$ satisfies the DML(1) property.
\end{corollary}

\subsection*{Acknowledgements}
We thank the anonymous referee for some beneficial suggestions.

This work is supported by the National Natural Science Foundation of China Grant No. 12271007.

\section{Proof of the main results}

Firstly, we recall a lemma which plays a key role in the proof. It guarantees that for a rational function which has no exceptional points, any subsequence of a non-preperiodic orbit cannot tend to any point with the maximal speed. It was proved in \cite[Theorem E]{Sil93}. See also \cite[Theorem 1.11]{Mat23} and \cite[Theorem 1.8]{Mat25} for some generalizations.

\begin{lemma}\label{Sil}
Let $f:\mathbb{P}^1 \to \mathbb{P}^1$ be an endomorphism of degree $d$ over a number field K, and $p\in \mathbb{P}^1(K)$ be a non-exceptional point. Fix a coordinate of $\mathbb{P}^1$ such that $p$ is not the infinity. Let $v\in M_{K}$ be a place. Then for any non-preperiodic $x\in \mathbb{P}^1(K)$, we have
    $$
        \lim_{n \to +\infty} \frac{-\log \min \{|f^n(x)-p|_{v}, 1\}}{d^n} = 0.
    $$
\end{lemma}

Now we can prove Proposition \ref{mainprop}.

\proof[Proof of Proposition \ref{mainprop}]
    First we make some assumptions and reductions.

    In view of \cite{BGT10} and \cite[Corollary 1.9]{XY}, we only need to treat the case when $\deg f = \deg g = d > 1$.
    
    If $g$ is conjugated to a polynomial map, then we reduce to the case of endomorphisms of $\mathbb{A}^2$ \cite{Xie17}. Since the DML property is invariant under iteration, we may assume that any iteration of $g$ is not conjugated to a polynomial map. Let $(x_0,y_0)$ be the starting point and let $C\subseteq\mathbb{A}^1\times\mathbb{P}^1$ is an irreducible curve. In order to prove the DML conjecture, we may assume that $x_0$ and $y_0$ are neither preperiodic points for $f$ and $g$, and that $C$ is not a fiber of $\mathbb{A}^1$ or $\mathbb{P}^1$. In this case, we prove that $C(\overline{\mathbb{Q}}) \cap O_{f \times g}(x_0,y_0)$ is finite.
    
    Extend $f$ to $\infty$. Let $\overline{C}$ be the closure of $C$ in $\mathbb{P}^1\times\mathbb{P}^1$ and $l_\infty = \{\infty\} \times \mathbb{P}^1$. Let $b_1, \dots, b_k$ be the second factor of the intersection points in $\overline{C} \cap l_\infty$. Applying a suitable conjugation by an element in $\mathrm{Aut}(\mathbb{P}^1)$ on the second $\mathbb{P}^1$-factor, we may assume that $\infty\notin\{b_1,\dots,b_k\}$. Let $K$ be a number field so that all data above are defined over $K$.

\begin{lemma}\label{norm growth}
    If $x_0 \in \mathbb{A}^1(K)$ is not a preperiodic point of f, then there exist a place $v \in M_K$, constants $c_1,c_2>0$, and a positive integer $N$, such that for $n>N$, we have
    $$ c_1 d^n < \log |f^n(x_0)|_v < c_2 d^n. $$
\end{lemma}

\begin{proof}
    Write $f = a_d x^d + a_{d-1} x^{d-1} + \cdots + a_0$, where $a_d \neq 0$. Denote $M_{K,\infty}$ as the set of archimedean places of $K$. Let $S = M_{K,\infty}\cup\{ v\in M_K \big{|}\ |a_d|_v \neq 1\}\cup\bigcup\limits_{i=0}^{d-1}\{ v\in M_K \big{|}\ |a_i|_v > 1\}$. Note that $S$ is a finite set. For $v\in S$, we denote $C_v=\frac{2}{|a_d|_v}(1+\sum\limits_{i=0}^{d-1}|a_i|_v)+1$.
    
    Since $O_f(x_0)$ is infinite, we know $\{h(f^n(x_0)) |\ n\in \mathbb{N}\}$ is unbounded by the Northcott property. Here $h$ is the height function. Then we can find either a place $v\notin S$ together with an integer $N$ such that $|f^N(x_0)|_v > 1$, or a place $v\in S$ together with an integer $N$ such that $|f^N(x_0)|_v > C_v$.
    
    In the previous case, we have $|f^{n+1}(x_0)|_v = |f^n(x_0)|_v^d$ when $n\geq N$. Hence the lemma follows.
    
    In the latter case, the inequalities $\frac{1}{2} |a_d|_v < \frac{|f^{n+1}(x_0)|_v}{|f^n(x_0)|_v^d} < \frac{3}{2} |a_d|_v$ and $|f^n(x_0)|_v>C_v$ hold for every $n\geq N$. Taking logarithm, we get
    $$ \log(\frac{1}{2}|a_d|_v) < \log |f^{n+1}(x_0)|_v - d\log|f^n(x_0)| < \log(\frac{3}{2}|a_d|_v).$$
    
    Then for $n\geq N$, by taking summation in the standard way, we get $$(\log|f^N(x_0)|_v+\frac{\log(|a_d|_v/2)}{d-1})d^{n-N} - \frac{\log(|a_d|_v/2)}{d-1} < \log |f^n(x_0)|_v$$
    $$< (\log|f^N(x_0)|_v+\frac{\log(3|a_d|_v/2)}{d-1})d^{n-N} - \frac{\log(3|a_d|_v/2)}{d-1}.$$
    Thus we finish the proof.
\end{proof}

    Now assume that $C \cap O_{f \times g}((x_0,y_0))$ is infinite. By Lemma \ref{norm growth}, we find a place $v\in M_K$ where $|f^n(x_0)|_v \to \infty$. Let $(n_l)_{l\geq1}$ be the return set $\{n\in \mathbb{N}|\ (f^{n}(x_0), g^{n}(y_0)) \in C(K)\}$, $\varphi$ be the defining function of $C$, and $\varphi_\infty = \varphi|_{l_\infty}$. Let $x$ and $y$ indicate the standard coordinates on $\mathbb{A}^1$ and $\mathbb{P}^1$, respectively. Write $\varphi = \sum_{i=0}^{m}{x^i \sum_{j=0}^{n}{a_{ij}y^j}}$, then $\varphi_\infty = \sum_{j=0}^{n}{a_{mj}y^j}$. Then $a_{mn} \neq 0$ as we have assumed that $\overline{C}$ does not intersect $l_\infty$ at $(\infty,\infty)$.

    For a point $(x,y) \in C(K)$ such that $x\neq0$, by rearranging the terms of the defining equation $\varphi$ and by dividing both sides by $x^m$, we have
    \begin{equation}\label{eq1}
        \sum_{i=0}^{m-1}{x^{i-m} \sum_{j=0}^{n}{a_{ij}y^j}} = -\sum_{j=0}^{n}{a_{mj}y^j} = -a_{mn}\prod_{s=1}^{k}(y-b_s)^{l_s}
    \end{equation}
    where $l_1,\dots,l_k$ are the multiplicities of the roots $b_1,\dots,b_k$ in $\varphi_{\infty}$.

    Write $(x_l,y_l) = (f^{n_l}(x_0), g^{n_l}(y_0))$. We claim that $|x_l|_v \to \infty$ forces $\min\limits_{1\leq s \leq k}|y_l - b_s|_v \to 0$. Otherwise, by extracting subsequence, we can assume $|y_l-b_s|_v>\varepsilon$ for some $\varepsilon>0$, every $s$ and every $l\geq 1$. If $\{|y_l|_v\}_{l\geq 1}$ is bounded, then when $l\to\infty$, the LHS of (\ref{eq1}) tends to 0 while the RHS has a positive lower bound, which is impossible. If $\{|y_l|_v\}_{l\geq 1}$ is unbounded, by extracting subsequence, we assume $|y_l|_v\to\infty$. Divide by $y_l^n$ in the both sides of (\ref{eq1}), then when $k\to\infty$, the LHS of (\ref{eq1}) tends to 0 while the RHS tends to $a_{mn}\neq 0$, a contradiction. Therefore, we get $\min\limits_{1\leq s \leq k}|y_l - b_s|_v \to 0$.

    Passing to subsequence, we may assume $|y_l - b_1|_v \to 0$. Then there is a constant $c_0$ such that $|y_l-b_1|_v^{l_1} < \frac{c_0}{|x_l|_v}$ for $l$ sufficiently large. So we get $-\log |y_l - b_1|_v > cd^{n_l}$ for some constant $c>0$ when $l$ is large.
    
    Now we apply Lemma \ref{Sil} to get a contradiction. It only remains to verify that $b_1$ is not an exceptional point for $g$. But as we have assumed that no iterate of $g$ can conjugate to a polynomial map, in fact $g$ has no exceptional point. Otherwise, the iteration $g^2$ will have an invariant exceptional point. Applying a suitable conjugation by an element in $\mathrm{Aut}(\mathbb{P}^1)$, we may send that point to $\infty$ and then $g^2$ is conjugated to a polynomial map. Thus we finish the proof.
\endproof

Now we prove Theorem \ref{mainthm} and Corollary \ref{cor2}.

\proof[Proof of Theorem \ref{mainthm}]
    We may assume that $f$ and $g$ are dominant. By taking normalization, we assume that $X$ and $Y$ are smooth. Take a smooth projective closure $\bar{X}$ of $X$. Then we can extend $f$ to an endomorphism $\bar{f}: \bar{X}\to\bar{X}$. 

    If the genus of $\bar{X}$ is greater than $1$, then some iteration of $\bar{f}$ is the identity, and the DML conjecture holds trivially in this case. The same is true if the genus of $Y$ is greater than $1$.

    If the genus of $\bar{X}$ and $Y$ are both $1$, then $\bar{f}$ and $g$ are both étale. Hence $\bar{f}\times g$ is also étale. Then the DML conjecture holds by \cite{BGT10}.

    If the genus of $\bar{X}$ is $1$ and the genus of $Y$ is $0$, then $\bar{f}$ is étale and the DML conjecture holds by \cite[Corollary 1.2]{BZ23}. The same is true if genus of $\bar{X}$ is $0$ and the genus of $Y$ is $1$.

    If the genus of $\bar{X}$ and $Y$ are both $0$, then $Y\cong \mathbb{P}^1$ and $X\cong \mathbb{P}^1 \setminus E$, where $E$ is a non-empty finite set. Since $\bar{f}$ is surjective, we have $\bar{f}(E) = E$, which implies that every point in $E$ is periodic. After iteration, we may assume that they are all fixed points. Then we can extend $f$ to $\bar{X}\setminus \{ \text{one point}\} \cong \mathbb{A}^1$, and the result follows from Proposition \ref{mainprop}.
\endproof

\proof[Proof of Corollary \ref{cor2}]
    Let $C\subset X\times Y$ be a curve, $(x,y) \in (X\times Y)(\overline{\mathbb{Q}})$ be a point, and $p_1: X \times Y \to X$, $p_2: X \times Y \to Y$ be projections. In order to  verify the DML(1) property, we may assume that both $x$ and $y$ are not preperiodic. Suppose $\# O_{f\times g}((x,y))\cap C(\overline{\mathbb{Q}}) = \infty$.  Let $C_1:=p_1(C)$ and $C_2:=p_2(C)$. Then $C_1\subset X$ is a curve, and $\# O_f(x)\cap C_1(\overline{\mathbb{Q}}) = \infty$. By the DML(1) property for $f$, there are positive integers $n_0$ and $m$ such that the infinite sequence $\{f^{n_0+km}(x)|\ k\in \mathbb{N}\} \subset C_1(\overline{\mathbb{Q}})$. Hence $f^m(C_1)=C_1$. Similarly, $C_2$ is a periodic curve for $g$.

    After iteration, we may assume that $C_1$ is invariant for $f$ and $C_2$ is invariant for $g$, i.e. $f(C_1)=C_1$ and $g(C_2)=C_2$. Then it suffices to verify the DML conjecture for $f|_{C_1} \times g|_{C_2}$ and $C\subset C_1\times C_2$, which follows from Theorem \ref{mainthm}.  
\endproof

\bibliographystyle{alpha}
\bibliography{reference}

@unpublished{XY,
    title={Height arguments toward the dynamical {M}ordell--{L}ang problem in arbitrary characteristic},
    author={J. Xie and S. Yang},
    note={arXiv:2504.01563v2}
}

@book{BGT16,
    title={The Dynamical Mordell--Lang Conjecture},
    author={J. P. Bell and D. Ghioca and T. J. Tucker},
    publisher={American Mathematical Society},
    year={2016},
    series={Mathematics Surveys and Monographs},
    volume={\textbf{210}},
    address={Providence, RI}
}

@article{Bel06,
    title={A generalized {S}kolem--{M}ahler--{L}ech theorem for affine varieties},
    author={J. P. Bell},
    journal={J. London Math. Soc. (2)},
    year={2006},
    volume={\textbf{73}},
    number={2},
    pages={367--379}
}

@article{BGT10,
    title={The dynamical {M}ordell--{L}ang problem for \'etale maps},
    author={J. P. Bell and D. Ghioca and T. J. Tucker},
    journal={Amer. J. Math.},
    volume={\textbf{132}},
    number={6},
    year={2010},
    pages={1655--1675}
}

@article{Xie17,
    title={The dynamical {M}ordell--{L}ang conjecture for polynomial endomorphisms of the affine plane},
    author={J. Xie},
    journal={Ast\'erisque},
    year={2017},
    volume={\textbf{394}},
    pages={vi+110}
}

@unpublished{Xie,
    title={Around the dynamical {M}ordell--{L}ang conjecture},
    author={J. Xie},
    note={Text available at http://scholar.pku.edu.cn/sites/default/files/xiejunyi/files/arounddml20230701fu\_ben\_.pdf}
}

@article{BZ23,
    title={$p$-adic interpolation of orbits under rational maps},
    author={J. P. Bell and X. Zhong},
    journal={Proc. Amer. Math. Soc.},
    volume={\textbf{151}},
    number={11},
    pages={4661--4672},
    year={2023}
}

@article{Mat25,
    title={Existence of arithmetic degrees for generic orbits and dynamical {L}ang--{S}iegel problem},
    author={Y. Matsuzawa},
    journal={J. Reine Angew. Math.},
    volume={\textbf{2025}},
    number={825},
    pages={305--335},
    year={2025}
}

@article{GT09,
    title = {Periodic points, linearizing maps, and the dynamical {M}ordell--{L}ang problem},
    author = {D. Ghioca and T.J. Tucker},
    journal = {J. Number Theory},
    volume = {\textbf{129}},
    number = {6},
    pages = {1392--1403},
    year = {2009}
}

@article{Sil93,
    title={Integer points, {D}iophantine approximation, and iteration of rational maps},
    author={J. H. Silverman},
    journal={Duke Math. J.},
    volume={\textbf{71}},
    number={3},
    pages={793--829},
    year={1993}
}

@article{Mat23,
    title={Growth of local height functions along orbits of self-morphisms on projective varieties},
    author={Y. Matsuzawa},
    journal={Int. Math. Res. Not. IMRN},
    number={4},
    pages={3533--3575},
    year={2023}
}

\end{spacing}
\end{document}